\documentclass[11pt]{article}
\usepackage{amsmath, amsfonts, amssymb}
\usepackage{mathtools}
\usepackage{graphicx, amsthm,color}
\usepackage{mathrsfs,dsfont}  
\usepackage{centernot}

\usepackage[T1]{fontenc}
\usepackage{fbb}

\usepackage{comment}
\usepackage[numbers,sort&compress]{natbib}
\usepackage{hyperref}

\usepackage{tikz}
\usepackage{tikz-3dplot}
\usetikzlibrary{calc,patterns,angles,quotes}
\usetikzlibrary{babel}
\usepackage{pgfplots}
\usepgfplotslibrary{patchplots}
\usetikzlibrary{backgrounds, intersections}
\usepgfplotslibrary{fillbetween}
\usetikzlibrary{arrows,automata}
\usepackage{subcaption}

\pgfplotsset{compat=1.18}
\usepgfplotslibrary{colorbrewer}

\newtheorem{theorem}{Theorem}[section]

\newtheorem{corollary}[theorem]{Corollary}
\newtheorem{proposition}[theorem]{Proposition}

\newtheorem{example}[theorem]{Example}
\newtheorem{claim}[theorem]{Claim}

\theoremstyle{definition}
\newtheorem{remark}[theorem]{Remark}

\usepackage[utf8]{inputenc}

\newcommand{\N}{\mathbb{N}}

\newcommand{\R}{\mathbb {R}}

\newcommand{\dist}{\textup{dist}}

\date{ \today }
\title{Sharp regularity results \\ for eikonal equations in Banach spaces} 
\author{\textsc{Tr\'i Minh L\^e \& Sebasti\'an Tapia-Garc\'ia}}

\begin{document}

% \maketitle

\begin{center}
{\LARGE Sharp regularity results \\ for eikonal equations in Banach spaces}
\end{center}
% \smallskip
\begin{center}
{\Large \textsc{Tr\'i Minh L\^e \& Sebasti\'an Tapia-Garc\'ia}}
\end{center}
\bigskip

\noindent\textbf{Abstract.} 
Viscosity solutions of eikonal equations need not enjoy regularity beyond Lipschitz continuity.
In contrast, differentiable solutions exhibit stronger structural properties.
In the Euclidean setting, building on observations of Caffarelli and Crandall (Comm. Partial Differential Equations $\textbf{35}$ (2010), 391--414), one finds that differentiable solutions of the eikonal equation $\|Df \| \equiv 1 \,\, \text{in } \mathcal{U} \subset \mathbb{R}^d$ are of class $C^{1, 1}_{\mathrm{loc}}(\mathcal U)$. 
Moreover, if $\mathcal{U} = \mathbb{R}^d$, then $f$ must be affine. 
This manuscript revisits these phenomena in the setting of (finite-- and infinite--dimensional) Banach spaces.
Our results show how the geometry of the underlying Banach space determines the regularity and affineness of Gateaux--differentiable solutions of eikonal equations.
As a geometric application in finite dimensions, we characterize differentiable functions whose gradient ranges are constrained to boundaries of convex bodies.
\bigskip


\vspace{0.6cm}

\textit{Secondary}: 46B20, 52A20, 35R15, 46G05.

% \tableofcontents

%%%%%%%%%%%%%%%%%%%%%%%%%%%%%%%%%
\section{Introduction}

The eikonal equation is one of the simplest Hamilton–Jacobi equations, yet the structure of its differentiable solutions is far from elementary. 
In the Euclidean setting, differentiable solutions may exhibit unexpected higher regularity and entire solutions are forced to be affine.
It is not clear, however, whether this phenomenon survives beyond Euclidean geometry, particularly in infinite-dimensional spaces where local compactness is unavailable. 
This work aims to identify geometric assumptions under which such higher regularity and affinity remain valid.

\medskip

A classical approach to regularity for Hamilton--Jacobi equations $H(x, Du) = 0$ relies on semiconcavity estimates for viscosity solutions. 
Under suitable assumptions on the Hamiltonian $H$, applying such estimates to both a solution $u$ and $-u$ yields $C^{1,1}$ regularity, see, for instance, \cite{F_2003,R_2008} and \cite[Theorem~5.3.6 and Corollary~3.3.8]{CS_2004}.
In particular, the eikonal equation $\|Du\|\equiv 1$ is also closely connected with the distance functions, whose regularity is governed by the cut locus and the nearest-point projection, see \cite{LN_2005,KS_2023}.
At lower regularity, however, the singular set of a distance function may be large or even dense (see \cite{S_2021}), highlighting the exceptional rigidity of differentiable solutions.

\medskip

A central insight of Caffarelli and Crandall  \cite{CC_2010} is that differentiable solutions of the Euclidean eikonal equation possess straight characteristics along which the solution is affine and its gradient remains constant. 
This geometric property implies that every differentiable solution is of class $C^{1,1}_{\mathrm{loc}}$. 
In the entire-space setting, the existence of complete characteristic lines further forces the solution to be affine. 
More recently, Ignat provided a short and elementary proof of the  $C^{1,1}_{\mathrm{loc}}$ regularity result \cite{I_2025}.

\medskip 

In contrast, the Banach space setting remains much less understood. 
Besides the loss of local compactness in infinite dimensions, the Hamiltonian $H(x, p)=\|p\|_*^2-1$ may fail to satisfy the smoothness and convexity assumptions underlying the classical semiconcavity approach. 
We note that the $C^{1, 1}$ regularity via the semiconcavity technique has also recently been developed for Hamilton--Jacobi--Bellman equations in separable Hilbert spaces under suitable control-theoretic assumptions, see \cite{DFSW_2025}.
For another perspective on the interaction between Banach space geometry and the eikonal equation via infinite games and a different notion of solution, see~\cite{DM_2007}.
In this context, it is natural to ask whether regularity and affineness of solutions can instead be derived directly from the geometry of the norm. 
Our results show that uniform convexity and smoothness properties of the unit sphere provide precisely the relevant geometric structure.

%\textbf{Notation:} Along this manuscript, $(X,\|\cdot\|)$ denotes a Banach space and $X^*$ its dual space, which is equipped with the canonical dual norm $\|\cdot\|_*$. $S_X$ stands for the unit sphere of $X$. The Gateaux-differential of a function $f$ is denoted by $D_Gf$
\subsection{Main contributions}

Our main result is Theorem~\ref{thm.straightline}, which provides a key geometric ingredient to analyze solutions of eikonal equations.
It shows that, under uniform convexity, every point determines a unique characteristic direction along which the solution is locally affine.

\begin{theorem}~\label{thm.straightline}
Let $X$ be a uniformly convex Banach space and let $\mathcal{U}\subset X$ be a nonempty open set. 
Let $f:\mathcal{U}\to \R$ be a Gateaux-differentiable function such that   
\begin{equation}
	\| D_G f(x) \|_\ast = 1 \quad \text{ for every $x \in \mathcal{U}$}.
\end{equation}
Then, for every $x\in \mathcal{U}$, there exists a unique vector $y_x \in S_X$ such that $D_Gf(x)(y_x) = 1$ and 
\begin{equation}\label{straightline}
    f(x + ty_x)  = f(x)  + t \quad \text{ for every } t\in \R~\text{s.t. }[x,x+ty_x]\subset \mathcal{U}.
\end{equation}
In particular, if $\mathcal U = X$, the identity~\eqref{straightline} holds for every $t \in \R$. 
\end{theorem}

As a consequence, we obtain the following regularity results. 
\begin{corollary}\label{corol.regu+linear}
    In the context of Theorem~\ref{thm.straightline}, the following assertions hold true:
    \begin{itemize}
        \item[(i)] Assume in addition that $\mathcal U = X$ and $\|\cdot\|$ is Gateaux-differentiable at $y_{\overline x}$ for some $\bar x \in X \setminus \{ 0 \}$.
        Then, $f$ is affine. 
        \item[(ii)] Assume in addition that $\| \cdot \| \in C^1(X \setminus \{ 0 \})$. Then $f$ is Fréchet-differentiable and of class $C^1$. 
        %Furthermore, if $\mathcal{U} = X$, then $f$ is affine.
        \item[(iii)] Assume in addition that $\| \cdot \| $ is uniformly smooth. Then $f \in C^{1, \omega}_{\mathrm{loc}} (\mathcal{U})$, where $\omega$ is the concave modulus of continuity of $D\|\cdot\|:S_X\to X^*$. 
        In particular, if $X$ is a Hilbert space, then $f \in C^{1, 1}_{\mathrm{loc}}(\mathcal{U})$.
    \end{itemize}
\end{corollary}

\begin{remark}
The conclusion of Theorem~\ref{thm.straightline} does not hold in nonreflexive Banach spaces.
Indeed, thanks to James's theorem \cite{J_1964}, every nonreflexive Banach space $X$ admits a linear map $x^*\in S_{X^*}$ that does not attain its supremum in the unit ball $\overline{B}_X$.
Therefore, for $f=x^*$ and $\varepsilon>0$, there is no $y_0\in S_X$ such that $f(ty_0)=t$ for all $|t| < \varepsilon$.
On the other hand, differentiability of the norm is a necessary condition for the validity of Corollary~\ref{corol.regu+linear} (and the subsequent Proposition~\ref{prop.fini-dim}). See Example~\ref{ex: nonlinear solution}.
\end{remark}

\begin{remark}
\emph{(i)}
The regularity in Corollary~\ref{corol.regu+linear}-\emph{(iii)} is sharp.
Indeed, $f(x) = \| x \|$ solves $\| Df \|_\ast \equiv 1$ in $X \setminus \{ 0 \}$; hence, if the norm is $C^{1,\beta}$, with $\beta\in (0,1)$, but not $C^{1,\gamma}$ for $\gamma>\beta$, there is a solution of an eikonal equation that is exactly $C^{1,\beta}$ smooth.
\medskip

\emph{(ii)}
We cannot expect $C^2$ regularity for solutions in general, even if the norm is $C^\infty$-smooth on $X \setminus \{0\}$. 
Indeed, in the euclidean setting, consider the set $S\subset\R^2$ given by $S:=[0,1]\times\{0\}$ and the distance function $f:=d(\cdot,S)$. It follows that $f\in C^1$ on $\R^2\setminus S$ and $\|Df(x)\|=1$ for all $x\in \R^2\setminus S$. However, $f$ is not twice differentiable on $\{0\}\times(0,+\infty)$.
\end{remark}

\begin{remark}
\emph{(i)}
For the entire solution, Corollary~\ref{corol.regu+linear}--\emph{(i)} may be compared with \cite[Theorem 3]{T_2019}, where affine rigidity is proved for $C^1$ solutions in reflexive locally uniformly convex Banach spaces with Fr\'echet differentiable norm on $X \setminus \{ 0 \}$.
Under the stronger assumption of uniform convexity, our result starts from mere Gateaux-differentiable solutions and requires only Gateaux-differentiability of the norm at some point.

\medskip

\emph{(ii)} 
The results in Corollary~\ref{corol.regu+linear}--\emph{(ii)(iii)} complement known regularity results for distance functions, since they provide a natural class of solutions to eikonal equations.
In uniformly convex and uniformly smooth Banach spaces with moduli of power type, differentiability is closely related to prox--regularity and yields higher regularity, see e.g.,~\cite[Theorem 4.9]{BTZ_2006} and~\cite[Theorem 2.3]{BTZ_2010}.
In finite--dimensional Minkowski spaces, the distance from an arbitrary closed set is of class $C^{1, 1}_{\mathrm{loc}}$ away from the closure of its singular set, see~
\cite[Theorem 1.1 and 1.4]{KS_2023} and also \cite{NT_2025}.
Corollary~\ref{corol.regu+linear}-\emph{(ii)(iii)} not only applies to solutions of eikonal equations, but also yields regularity without imposing any power-type assumptions on moduli of smoothness and convexity of the space.
\end{remark}

In the finite-dimensional setting, as can be seen in Proposition~\ref{prop.fini-dim}, we can drop the assumption of uniform convexity on $\|\cdot\|$.
Indeed, we can circumvent it with a compactness argument.
We point out that a version of the second part of Proposition~\ref{prop.fini-dim} can be found in \cite{CC_2010}, under the extra assumption of uniform convexity.
Observe that the first part of Proposition~\ref{prop.fini-dim} falls within the classical approach using semiconcave estimates to obtain regularity results for Hamilton-Jacobi equations, see e.g. \cite[Theorem 5.3.7]{CS_2004}.\\

Since in the finite-dimensional setting Gateaux-differentiability and Fréchet-differentiability coincide for Lipschitz functions, we just say that a function is differentiable.

\begin{proposition}\label{prop.fini-dim}
    Let $\| \cdot \|$ be a differentiable norm on $\R^d\setminus\{0\}$ and let $\mathcal U \subset \R^d$ be a nonempty open set.
    Let $f: \mathcal U \to \R$ be a differentiable function such that 
    \[
        \| Df(x) \|_\ast = 1 \quad \text{ for every } x \in \mathcal U.
    \]
    Then, the following assertions hold true:
    \begin{itemize}
        \item[(i)]$f \in C^{1, \omega}_\mathrm{loc}(\mathcal U)$, where $\omega$ is the concave modulus of continuity of $D \| \cdot \|$ on the unit sphere.

        \item[(ii)] If $\mathcal U  = \R^d$, then $f$ is affine.
    \end{itemize} 
\end{proposition}
%\begin{remark}\label{rem.asymmetric} \ST{In fact, not so clear.}
%    In fact, Proposition~\ref{prop.fini-dim}~$ii)$ holds true in the more general framework of Minkowski gauges (or asymmetric norms). 
 %   Indeed, if $Q\subset \R^d$ is a compact convex body, with $0\in \mathrm{int}\,Q$ and smooth boundary $\partial Q$, and $f:\R^d\to\R$ is a differentiable function such that $\rho_Q^*(D f(x))=1$ for all $x\in \R^d$, where \[\rho_Q^*(z):=\sup_{x\in Q}\langle z,x\rangle,\quad\text{for all }z\in \R^d,\] then $f$ is affine.
%\end{remark}
On the other hand, the symmetry of the norm can be dropped, and it can be shown that the statement of Proposition~\ref{prop.fini-dim} $(ii)$ holds in the more general framework of Minkowski gauges, see Proposition~\ref{prop.asym}.
This new formulation readily implies the following result.

\begin{corollary}\label{cor.convexaffine}
    Let $C\subset \R^d$ be a compact convex body. Then, the following assertions are equivalent.
    \begin{itemize}
        \item[(1)] $\partial C$ does contain a nontrivial segment. 
        \item[(2)] There is a non--affine differentiable map $f:\R^d\to \R$ such that $\nabla f(\R^d)\subset \partial C$.  
    \end{itemize}
\end{corollary}
%\begin{example} (nonsmooth norm and nonlinear solutions) %\normalfont
%    If one considers a nonsmooth norm, then there may exist nonlinear $C^\infty$ solutions to the equation $\| Df \|_* \equiv 1$.
%    Indeed, consider $\mathbb R^2$ with $1$-norm (and thus $(\mathbb{R}^2)^*$ is equipped with the $\infty$-norm). 
%    Fix $\varphi \in C^\infty(\mathbb R)$ (or even just differentiable, not $C^1$) such that $\varphi'(\mathbb R) \subset [-1, 1]$.
%    Set $f(x, y) = x + \varphi(y)$ for every $x, y \in \mathbb R$.
%    A direct computation gives $f_x = 1 $ and $f_y = \varphi'(y)$ and hence $\|Df\|_*=\max\{|f_x|,~|f_y|\} = 1$ in $\mathbb R^2$.
%    For instance, one can take $\varphi(y) = \sin(y)$.
%\end{example}

The last result that follows from our technique is a representation formula for solutions.
When the norm is not Gateaux-differentiable, we cannot, in general, conclude that entire solutions are affine. 
Nevertheless, the following corollary shows that such solutions are signed distance functions to each of their level sets. 
This representation formula has already been discussed in \cite{CC_2010} in the finite-dimensional setting.
For a function $f:\mathcal{U}\to \R$ and $r\in \R$ we write $[f=r]:=\{x\in \mathcal{U}:~f(x)=r\}$.
\begin{corollary}\label{corol.dist.2} 
 In the context of Theorem~\ref{thm.straightline}, let $x\in \mathcal{U}$ and $r< f(x)$ such that there exists ${z_x\in [f=r]}$ satisfying
 \begin{equation}\label{eq.cor.dist}
   \|x-z_x\|= \mathrm{dist}(x, [f=r])\quad\text{and}\quad [x,z_x]\subset \mathcal{U}.  
 \end{equation}
 Then, 
\begin{align*}
    & \, f(x) = r + \mathrm{dist}(x, [f = r]).
\end{align*} 
In particular, if $\mathcal{U}=X$, then for any $x\in X$ and $r\in \R$
\begin{align*}
    & \, f(x) = r + \mathrm{dist}(x, [f = r]), \quad \text{ if } f(x) \geq r, \\
    & \, f(x) = r - \mathrm{dist}(x, [f = r]), \quad \text{ if } f(x) \leq r.
\end{align*}   
\end{corollary}
\begin{comment}
\begin{corollary}\label{corol.dist}{\normalfont (representing via distance functions)}
In the context of Theorem~\ref{thm.straightline}, assume that $\mathcal U = X$.
Then, the following relations hold true
\begin{align*}
    & \, f(x) = r + \mathrm{dist}(x, [f = r]), \quad \text{ if } f(x) \geq r, \\
    & \, f(x) = r - \mathrm{dist}(x, [f = r]), \quad \text{ if } f(x) \leq r.
\end{align*}
\end{corollary}
\end{comment}

\subsection{Preliminaries in Banach space theory}\label{sec: preliminaries}
In this section we recall some concepts from Banach space theory that will be needed throughout this work. 
From now on, $(X,\|\cdot\|)$ always denotes a real Banach space and $(X^*,\|\cdot\|_*)$ is its dual space. The evaluation of a linear map $x^*\in X^*$ at some point $x\in X$ is indistinctively written by $x^*(x)$ or $\langle x^*,x\rangle$.
$B_X$, $\overline{B}_X$ and $S_X$ denote the open unit ball, the closed unit ball and the unit sphere of $X$ respectively. 
We also write $S^{d-1}$ for the unit sphere of $\R^d$ (where the norm is understood from the context).
A function $f:X\to\R$ is Gateaux-differentiable at $\overline{x}\in X$ if there is $x^*\in X^*$ such that
\[\lim_{t\to 0}\frac{f(\overline{x}+ty)-f(\overline{x})}{t}-x^*(y)=0,\quad \text{for all }y\in X.\]
We write $D_Gf(\overline{x}):=x^*$. 
Further, $f$ is Fréchet-differentiable at $\overline{x}$ if 
\[\lim_{t\to 0}\sup_{y\in B_X}\left| \frac{f(\overline{x}+ty)-f(\overline{x})}{t}-x^*(y) \right| = 0.\]
We write $Df(\overline{x}):=x^*$. In particular, if $f$ is Fréchet-differentiable at $\overline{x}$, then it is Gateaux-differentiable at $\overline{x}$ with $D_Gf(\overline{x})=Df(\overline{x})$. Recall that for Lipschitz functions defined on finite-dimensional spaces, both notions of differentiability coincide.

\medskip

Concerning the geometry of $X$, we require the following three definitions: $X$ is said to be \textit{uniformly convex} if for every $\varepsilon\in(0,2)$ there is $\delta>0$ such that
    \[\text{for all } x,y\in X,~ \|x\|=\|y\|=1,\qquad \|x-y\|\geq \varepsilon\implies \|x+y\| \leq 2(1-\delta).\]
$X$ is said to be \textit{smooth} if $\|\cdot\|$ is Fréchet-differentiable outside the origin (and so ${\|\cdot\|\in C^1(X\setminus\{0\})}$) and \textit{uniformly smooth} if for every $\varepsilon\in(0,2)$ there is $\delta>0$ such that
    \[\text{for all } x,y\in X \,\text{ with }\|x\|=1,\|y\|\leq \delta,\, \text{ one has } \quad \|x+y\|+\|x-y\|\leq 2 + \varepsilon \|y\|.\]   
It is well known that uniformly smooth spaces have continuously differentiable norms (outside the origin). 
For the sake of brevity, whenever $\| \cdot \|$ is Fr\'echet differentiable, we write $J:X\setminus\{0\}\to X^*$ the map defined by $J(x):= D\|\cdot\|(x)$.

\begin{proposition}{\normalfont \cite[Fact 9.7]{Fabianbook}}
    Let $(X,\|\cdot\|)$ be a Banach space.
    Then $X$ is a uniformly smooth Banach space if and only if $\|\cdot\|$ is Fréchet-differentiable outside $0$ and $J: S_X\to X^*$ is uniformly continuous. 
\end{proposition}

Whenever $J |_{S_X}$ is uniformly continuous, we denote by $\widehat{\omega} : [0,+\infty)\to \R\cup\{+\infty\} $ the modulus of continuity of $J$ on $S_X$.
That is
\[\widehat{\omega}(t):=\sup\{\|J(x)-J(y)\|_*:~x,y\in S_X,~\|x-y\|\leq t\},\quad\text{for all }t\geq 0.\]
The \textit{concave modulus of continuity} of $J$ is denoted by $\omega : [0,+\infty)\to\R\cup\{+\infty\}$ and it is the concave envelope of $\widehat{w}$.
It is well known that $\lim_{t\to 0^+}\omega(t)=0$.

\medskip

For an open set $\mathcal{U}\subset X$, we say that a function $f\in C^{1,\omega}_{loc}(\mathcal{U})$ if $f$ is continuously differentiable and for every $x\in \mathcal{U}$, there is $r>0$ and $C>0$ such that
\[\|Df(y)-Df(z)\|_*\leq C\omega(\|y-z\|),\quad\text{for all }y,z\in B(x,r)\cap\mathcal{U}.\]
\textbf{Outline of the manuscript.} In Section~\ref{sec.2} we present the proof of Theorem~\ref{thm.straightline}. In Section~\ref{sec.3} we continue with the proofs of Corollary~\ref{corol.regu+linear} and Corollary~\ref{corol.dist.2}. Section~\ref{sec.4} is dedicated to the finite-dimensional setting, so we provide the proofs of Proposition~\ref{prop.fini-dim} and Corollary~\ref{cor.convexaffine}. Finally, we conclude this manuscript with an appendix that contains a useful regularity estimate.

\section{Proof of Theorem~\ref{thm.straightline}}\label{sec.2}
Fix the function $f:\mathcal{U}\subset X\to \R$ as in Theorem~\ref{thm.straightline}.
Without loss of generality, assume that $ \overline{B}_X\subset \mathcal{U}$ and that $f(0)=0$.

\medskip

\begin{claim}\label{claim.1lip}
    The function $f$ is $1$-Lipschitz in $\overline{B}_X$.
\end{claim}
\textit{Proof of Claim~\ref{claim.1lip}}
    Let $x,y\in \overline{B}_X$.
    Note that $[x, y] \subset \overline{B}_X$.
    Define
    \[
    g(t) := f(x + t(y - x)), \quad t \in [0, 1].
    \]
    Then, $g$ is differentiable and, by the chain rule,
    \[
        |g'(t)| = |D_Gf(x + t(y -x))(y - x)| \leq \| D_Gf(x + t(y -x)) \|_\ast \| y - x \| = \| y - x \|.
    \]
    Therefore, $|f(x)-f(y)| = |g(1)-g(0)| \leq \|y-x\|$ and so $f$ is $1$-Lipschitz.
\hfill$\Diamond$ 
% The next lemma is a key step of our argument.
\begin{claim}\label{claim.maxmin}
    One has
    \[
    \sup_{z \in \overline{B}_X} f(z) = 1 \,\, \text{ and } \,\, \inf_{z \in \overline{B}_X} f(z) =-1.
    \]
    Moreover, the above supremum and infimum are attained on $S_X$. 
\end{claim}

\textit{Proof of Claim~\ref{claim.maxmin}.}
Observe first that $\sup_{z \in \overline{B}_X} f(z)\leq 1$ since $f$ is $1$--Lipschitz and $f(0) = 0$.
Fix $\varepsilon \in (0, 1)$.
We prove that 
\[
    \sup_{z \in \overline{B}_X} f(z) \geq 1 - \varepsilon.
\]

\medskip

To this end, we construct a generalized sequence $\{ x_\lambda \}_{\lambda \leq \Lambda}$ starting from the origin and moving at each step in a direction along which $f$ increases with slope at least $1-\varepsilon$.
We continue this construction until we reach the boundary of $\overline B_X$.
Since the total increase of $f$ along the construction dominates $(1-\varepsilon)$ times the traveled distance, the final point will satisfy $f(x_\Lambda)\ge 1-\varepsilon$.

\medskip

\textbf{Construction of $\pmb{\{ x_\lambda \}_{\lambda \leq \Lambda} }$.}
Define $x_0=0\in X$ and, proceeding by transfinite induction, we construct a generalized sequence $\{x_\lambda\}_{\lambda\leq \Lambda}$ that satisfies
    \begin{itemize}
        \item $f(x_{\lambda +1})-f(x_\lambda)> (1-\varepsilon)\|x_{\lambda+1} - x_\lambda\|$, for all $\lambda<\Lambda$,
        \item $x_\alpha=\lim_{\lambda\to \alpha} x_\lambda$, for all $\alpha\leq \Lambda$ limit ordinal.
        \item $x_\Lambda\in S_X$. 
    \end{itemize}
\textit{Successor step.}
Let $\lambda$ be an ordinal such that $x_\lambda\in \overline{B}_X$ is already defined. 
If $x_\lambda\in S_X$, the induction ends. 
If $x_\lambda\in B_X$, recalling that $\|D_G f(x_\lambda)\|_* = 1$, there exists $p_\lambda\in S_X$ such that
    \begin{align*}
        1-\varepsilon < D_G f(x_\lambda)(p_\lambda) = \lim_{t\to 0}\dfrac{f(x_\lambda+tp_\lambda)-f(x_\lambda)}{t}.
    \end{align*}
So, there is $t_\lambda>0$ such that $x_\lambda+t_\lambda p_\lambda\in \overline{B}_X$ and
    \[(1 -\varepsilon)t_\lambda  < f(x_\lambda+t_\lambda p_\lambda)-f(x_\lambda).\]
    Set $x_{\lambda+1}=x_\lambda +t_\lambda p_\lambda$.
\medskip

\textit{Limit step.}
Let $\alpha$ be a limit ordinal and assume that $\{x_\lambda\}_{\lambda<\alpha}\subset B_{X}$ is already defined. Then, thanks to the continuity of $f$ and the fact that $\{f(x_\lambda)\}_{\lambda<\alpha}$ is increasing, we get
    \begin{align*}
        \sum_{\lambda<\alpha} \|x_{\lambda+1}-x_\lambda\| \leq \sum_{\lambda <\alpha} \dfrac{f(x_{\lambda+1})-f(x_\lambda)}{1-\varepsilon}\leq \lim_{\lambda\to \alpha }\dfrac{f(x_\lambda)-f(x_0)}{1-\varepsilon}\leq \dfrac{1}{1-\varepsilon}.
    \end{align*}
    So, $\{x_\lambda\}_{\lambda<\alpha}$ is Cauchy net and therefore convergent.
    
    \medskip
    
    Observe that, since the generalized sequence $\{f(x_\lambda)\}_{\lambda \leq \alpha}$ is strictly increasing and $f$ is bounded on $\overline{B}_X$, the above induction process stops at a countable ordinal $\alpha:=\Lambda$, with $x_\Lambda\in S_X$.

    \medskip
    
    Now we realize that
    \[f(x_\Lambda)=\sum_{\lambda<\Lambda}f(x_{\lambda+1})-f(x_\lambda)\geq (1-\varepsilon)\sum_{\lambda<\Lambda}\|x_{\lambda+1}-x_{\lambda}\| \geq (1-\varepsilon)\|x_\Lambda-x_0\|=1-\varepsilon.\]
    So, $\sup_{z\in \overline{B}_X} f(z)\geq 1-\varepsilon$. Since $\varepsilon>0$ is arbitrary, we deduce that $\sup_{z \in \overline{B}_X} f(z) \geq 1$. Thus
    \[
        \sup_{z\in \overline{B_X}} f(z) = 1.
    \]
    By exchanging $f$ by $-f$, we deduce that $\inf_{z \in \overline{B}_X} f(z) = -1$.

    \medskip
    
    It remains to check that the above supremum is a maximum. Let $\{x_n\}_n,\{y_n\}_n\subset \overline{B}_X$ be such that 
    \[\lim_{n\to\infty} f(x_n)=1 \quad\text{and}\quad \lim_{n\to\infty}f(y_n)=-1.\]
    We can and shall assume that $\|x_n\|=\|y_n\|=1$ for all $n\in \N$. 
    We show that $\{x_n\}_n$ is a Cauchy sequence. 
    Fix $\varepsilon\in (0,2)$ and let $\delta:=\delta(\varepsilon)>0$ given by the uniform convexity of $X$. 
    
    \medskip
    
    Fix $k\in \N$ such that $f(y_k)<-1+\delta$. 
    Since $f$ is $1$-Lipschitz, we have that
    \[  f(x_n)-(-1+\delta) < f(x_n)-f(y_k)\leq \|x_n-y_k\|\]
    Therefore, there is $N_\delta\in \N$ such that
    \[ \|x_n-y_k\| \geq 2(1-\delta),\quad \text{for all }n\geq N_\delta.\]
    By definition of $\delta$, we have then that
    \[ \| x_n + y_k \| = \| x_n-(-y_k)\| \leq \varepsilon,\quad \text{for all }n\geq N_\delta. \]
    It follows that for any $m, n \geq N_\delta$
    \[
        \| x_n - x_m \| \leq \| x_n + y_k \| + \| x_m + y_k \| \leq 2\varepsilon.
    \]
Hence $\{x_n\}_n$ is Cauchy, and thus the sequence converges. 
Set $x_\infty:=\lim_{n\to\infty} x_n$. By continuity of $f$, we deduce that $f(x_\infty)=1$.
Analogously, by replacing $f$ by $-f$, we deduce that $y_\infty:=\lim_{n\to\infty}y_n$ is well defined and $f(y_\infty)=-1$.
Claim~\ref{claim.maxmin} is proven.
\hfill$\Diamond$

\medskip

\textbf{Conclusion.}
With the above claims established, we now complete the proof of Theorem~\ref{thm.straightline}.
We prove the result at the origin.
Let us check that $\mathrm{argmax}_{z \in \overline{B}_X}\,f(z)$ and $\mathrm{argmin}_{z \in \overline{B}_X}\,f(z)$ are singletons.
Indeed, otherwise, there would exist $z_1,z_2\in S_X$ such that $z_1\neq -z_2$, $f(z_1)=1$ and $f(z_2)=-1$.
Since $X$ is uniformly convex, we have $\| z_1 + (-z_2)\| < 2$.
Hence, we get
\[
    \dfrac{f(z_1) - f(z_2)}{\| z_1 - z_2\|} = \dfrac{2}{\| z_1 + (-z_2) \|}> 1,
\]
which contradicts the fact that $f$ is $1$-Lipschitz on $\overline{B}_X$.
Therefore both the maximizer and the minimizer of $f$ on $\overline{B}_X$ are unique.
Consequently, if one has
\[
\{ z_1 \} = \mathrm{argmax}_{z \in \overline{B}_X}\,f(z) \,\, \text{ and  } \,\, \{ z_2 \}= \mathrm{argmin}_{z \in \overline{B}_X}\,f(z),
\] 
then necessarily $z_1=-z_2$.
Let us write $y_0:=z_1$. 
Since $f$ is $1$-Lipschitz, we finally conclude that
    \[
    f(t y_0)= t,\quad\text{for all }|t|\leq1.
    \]
Now, by simple induction, repeating the argument at the endpoints of the interval of validity of the above identity, we finally deduce that 
\[
    f(t y_0)= t,\quad\text{for all }t\in \R~\text{such that }[0,ty_0]\subset \mathcal{U}.
    \]
\hfill$\square$

\section{Proofs of Corollaries}\label{sec.3}

\subsection{Proof of Corollary~\ref{corol.regu+linear}}

% For the sake of brevity, throughout the proof we write $J := D \| \cdot \|$.
%Since $X$ is uniformly convex, for every $x^*\in S_{X^*}$ there exists a unique $x\in S_X$ such that $x^*=J(x)$.

\emph{(i).}
The following argument is due to Crandall~\cite[Lemma~7.3(b)]{C_2008}, and in Banach spaces, it only requires Gateaux-differentiability of the norm at $y_{\bar x}$.
Since $\mathcal U = X$, Theorem~\ref{thm.straightline} gives
\begin{equation}\label{aff.along.line}
    f(\bar x + r y_{\bar x}) = f(\bar x) + r
    \qquad\text{for every }r\in\mathbb{R}.
\end{equation}
Set $x^* := D_G\|\cdot\|(y_{\bar x})\in X^*$.
Fix $z\in X$ and write $h := z-\bar x$. 
For every $r > 0$, the Lipschitz continuity of $f$ and  the identity~\eqref{aff.along.line} yield
\begin{equation}\label{aff.compa}
 r - \| r y_{\bar x} - h\|  \leq f(z) - f(\bar x) \leq - r + \| r y_{\bar x} + h\|.
\end{equation}
Since the norm is Gateaux-differentiable at $y_{\bar x}$, we get
\[
 \lim_{r \to + \infty} \bigl(\|r y_{\bar x} + h\| - r \bigr) = x^*(h)
 \quad \text{ and } \quad 
 \lim_{r\to+\infty}\bigl(r - \|r y_{\bar x} - h \|\bigr) = x^*(h).
\]
Thus, letting $r \to + \infty$ in~\eqref{aff.compa}, we obtain $f(z) - f(\bar x) = x^*(z-\bar x).$
Since $z \in X$ was arbitrary, we infer that
\[
    f(z) = x^*(z) + f(\bar x) - x^*(\bar x) \qquad\text{for every } z \in X.
\]
Therefore, $f$ is affine.

\medskip

\textit{(ii).}
Fix $x \in \mathcal U$.
We first show that $f$ is Fr\'echet differentiable at $x$. 
Thanks to Theorem~\ref{thm.straightline}, there exists $y \in S_X$ and $r > 0$ so that
\[
    D_G f(x)(y) = 1 \,\, \text{ and }  \,\, f(x + ty) = f(x) + t \,\, \text{ for every $t \in [-r, r]$.}
\]
Since $f$ is $1$-Lipschitz in a neighborhood of $x$ and $\| D_G f(x) \|_\ast = 1$, the directional derivative of $f$ at $x$ attains the Lipschitz constant in the direction $y$. 
Using the fact that $\| \cdot \| \in C^1(X \setminus \{ 0 \})$, it follows from the result of Fitzpatrick \cite[Theorem 2.4]{F_1984} that $f$ is Fr\'echet differentiable at $x$. 

\medskip

Now we prove that $Df$ is continuous in $\mathcal{U}$. Assume that $f$ is not $C^1$.
Without loss of generality, we can assume that $B_X \subset \mathcal U$, $f(0) = 0$, and $Df$ is not continuous at $0$.
Therefore, there is a sequence $\{x_n\}_n\subset X$, convergent to~$0$, and $\varepsilon>0$ such that $\|Df(x_n)-Df(0)\|_\ast \geq \varepsilon$ for all $n\in \N$.
By uniform convexity of $X$, there is a unique point $z_n\in S_X$ such that $Df(x_n)(z_n)=1$. 
By smoothness of $\|\cdot\|$, we also have $Df(x_n)=J(z_n)$.
After applying Theorem~\ref{thm.straightline} at $x_n$, we deduce that 
\[f(x_n +tz_n)= f(x_n)+t,~\quad \text{for all }|t|\leq 1-\|x_n\|.\]
Let us call $z_\infty \in S_X$ the unique point such that $Df(0)(z_\infty)=1$. 
Observe that there is $\eta>0$ such that $\|z_n-z_\infty\|>\eta$.
Indeed, otherwise we will have that \[\liminf_{n\to\infty} \|Df(x_n)-Df(0) \|_\ast= \liminf_{n\to\infty}\| J(z_n)-J(z_\infty)\|_\ast =0,\] due to the continuity of $J$ on $S_X$. 
Choose $t_n> 0$ such that $y_n:=x_n - t_nz_n\in S_X$.
It follows that $t_n\in[1-\|x_n\|,1+\|x_n\|]$ and $f(y_n)= f(x_n)-t_n$.
Now, note that
\begin{align*}
    \|y_n+z_\infty\| &\geq \|z_n-z_\infty\| - \|x_n+ (1 - t_n)z_n\|\geq \eta-2\|x_n\|.
\end{align*}
Let $N\in \N$ be such that $\|x_n\|\leq \eta/4$ and consider $\delta:= \delta(\eta/2)$ given by the uniform convexity of $X$. Then, 
\begin{align*}
    \|y_n - z_\infty\| \leq 2(1-\delta),\quad\text{for all }n\geq N.
\end{align*}
However, we notice that
\begin{align*}
    \limsup_{n\to\infty}\dfrac{f(z_\infty)-f(y_n)}{\|y_n-z_\infty\|}\geq \limsup_{n\to\infty}\dfrac{1 - (f(x_n)-t_n)}{2(1-\delta)}= \dfrac{1}{1-\delta}.
\end{align*}
The above expression contradicts the $1$-Lipschitzianity of $f$. 
Therefore, $ Df$ is continuous at $0$.

\medskip

\textit{(iii).} 
Fix $\bar x \in \mathcal U$ and let $r > 0$ be such that $B(\bar x, 2r) \subset \mathcal U$.
% We prove that there exist $\delta < r$ and $L > 0$ such that 
% \begin{equation}\label{df_holder}
%	\| Df(x) - Df(y) \|_\ast \leq L \omega(\| x - y \|), \quad \text{ for every } x, y \in B_X(\bar x, r) \text{ with } \| x - y \| \leq \delta.
%\end{equation}
Thanks to Theorem~\ref{thm.straightline}, for every $z \in B(\bar x, r)$, there exists a unique $y_z \in S_X$ such that 
\begin{equation}\label{linear_along_line}
Df(z) = J(y_z), \quad Df(z + t y_z) = Df(z) \quad \text{ and } \quad f(z + t y_z) = f(z) + t, 
\end{equation}
for every $t \in [- r, r]$.
Fix $\tau \in (0, r)$. 
Then  $x + \tau e \in \mathcal{U}$ for every $x \in B(\bar x, r) \text{ and } e \in S_X$.
Since $f$ is $1$--Lipschitz on $B(\bar x, 2r)$, we observe that 
\begin{equation}\label{upper_lip}
    f(x + h) \leq f(x - \tau y_x) + \| h + \tau y_x \| = f(x) - \tau + \| h + \tau y_x \| \quad \text{ for every } x, x + h \in B(\bar x, 2r).
\end{equation}
Similarly, we have
\begin{equation}\label{lower_lip}
    f(x + h) \geq f(x) + \tau - \| h - \tau y_x \| \quad \text{ for every } x, x + h \in B(\bar x, 2r).
\end{equation}
Denote $\mathcal{A} := \{ x \in X: \tau/2 \leq \| x \| \leq 3\tau/2 \}$.
Recall that $\omega$ is the concave modulus of continuity of $J$ on $S_X$.
It follows from the Mean Value Theorem that, for some $C_\tau > 0$,
\begin{equation}\label{taylor_norm}
    \big| \| a + h \| - \| a \| - J(a)(h)  \big| \leq C_\tau \| h \| \omega(\|h\|) \text{ for every $a, h$ with $[a, a + h] \subset \mathcal{A}$}.
\end{equation}
Notice that $J$ is $0$-homogeneous and hence, thanks to \eqref{linear_along_line}, it holds $J(\tau y_x) = J(y_x) = Df(x)$.
Combining \eqref{upper_lip} and \eqref{taylor_norm}, we obtain
\begin{align*}
    f(x + h) \leq & ~ f(x) - \tau + \| h + \tau y_x \| \\
    \leq & ~ f(x) - \tau + \Big( \| \tau y_x \| +  J(\tau y_x)(h)  + C_\tau \| h \| \omega(\|h\|) \Big) \\
    = & ~ f(x) +  Df(x)( h)+ 
    C_\tau \| h \| \omega(\| h \|),
\end{align*}
for every $x \in B(\bar x, r)$ and $\| h \| \leq \tau/2$.
Using \eqref{lower_lip} and \eqref{taylor_norm}, we get an analogous lower bound and therefore, for every $x, y \in B(\bar x, r)$ with $\| x - y \| \leq \tau/2$, one has
\begin{align*}
    \big| f(y) - f(x) -  Df(x)( y - x )  \big| \leq C_\tau \| y - x \| \omega(\|y - x \|).
  \end{align*}
Thanks to Proposition~\ref{prop.C1ome}, we conclude that $Df$ is uniformly continuous in $B(\bar x, \tau/8)$ with modulus $\widetilde\omega(t) = \widehat{C}_\tau \omega(t)$, for some $\widehat{C}_\tau > 0$.
\hfill$\square$

\begin{example} (nonsmooth norm and nonlinear solutions) \label{ex: nonlinear solution} \normalfont
    If one considers a norm which is not Gateaux-differentiable at some $\overline{x}\neq 0$, then the equation $\| Df \|_* \equiv 1$ admits nonlinear differentiable solutions which are not $C^1$-smooth.
    Indeed, let $(X,\|\cdot\|)$ be a Banach space such that there are $\overline{x}\in S_X$ and two vectors $x^*,y^*\in S_{X^*}$, with $x^\ast \neq y^\ast$, satisfying $x^*(\overline{x})=y^*(\overline{x})=1$.
    It follows that 
    \[[x^*,y^*]:=\mathrm{conv}\{x^*,y^*\}\subset S_{X^*}.\]
    Fix $\varphi :\mathbb R\to\R$ be an everywhere differentiable function which is  not of class $C^1$ such that $\varphi'(\mathbb R) \subset [-1/2, 1/2]$ (for instance, see \cite{DDT2024,DT2025}).
    Set \[f(x) =\frac{1}{2}\langle x^*+y^*,x\rangle + \varphi(\langle x^*-y^*,x\rangle),\quad\text{for every }x \in \mathbb X.\]
    A direct computation gives that for all $x\in X$,
    \[Df(x)= \dfrac{x^*+y^*}{2} + \varphi'(\langle x^*-y^*,x\rangle)(x^*-y^*)\in [x^*,y^*]\subset S_{X^*}.\]
    Therefore, $\|Df(x)\|_{*}=1$ for all $x\in X$. 
    The above construction is essentially the same as that in \cite[Remark 2.1]{CC_2010}, which was considered in the finite--dimensional setting.
\end{example}

\subsection{Proof of Corollary~\ref{corol.dist.2}}
Fix $x \in \mathcal U$ and $r < f(x)$.
Assume that there is $z_x\in [f=r]$ satisfying~\eqref{eq.cor.dist}.
Since $f$ is $1$--Lipschitz in $[x, z_x]$, it follows that
\[
    f(tx+ (1-t)z_x) \leq f(z_x)+t\|x-z_x\|=r+\mathrm{dist}(tx+ (1-t)z_x, [f = r]), \quad\text{for all }t\in[0,1].
\]
Let $t_0\in (0,1]$ be such that $x_0:= t_0x+(1-t_0)z_x$ satisfies $f(x_0) > r$ and $\overline{B}(x_0,\|x_0-z_x\|)\subset \mathcal U$. 
By definition of $z_x$, we claim that $\overline{B}(x_0,\|x_0-z_x\|)\subset [f\geq r]$.
Indeed, if not, there would exist $w \in \overline{B}(x_0,\|x_0-z_x\|)$ such that $f(w) < r$.
Since $f(w) < r < f(x_0)$, using the intermediate value theorem, there exists $w' \in [x_0, w] \setminus \{ x_0, w \} \subset \mathcal{U}$ such that $f(w') = r$.
It follows that
\[
    \| x - w' \| \leq 
    \|x - x_0\| + \| x_0 - w' \| < \| x - x_0 \| + \| x_0 - z_x \| = \| x - z_x \|,
\]
which contradicts the definition of $z_x$.
\medskip

Applying Theorem~\ref{thm.straightline} at $x_0$, there exists a unique direction $y_0\in S_X$ such that 
\[
    f(x_0 - t y_0) = f(x_0) - t \quad \text{ for every } t \in [0, \| x_0 - z_x \|].
\]
Since $x_0 - \| x_0 - z_x \| y_0 \in \overline{B}(x_0,\|x_0-z_x\|)$, we get $f(x_0) \geq r + \| x_0 - z_x \|$.
On the other hand, since $f$ is $1$--Lipschitz on $[x_0, z_x]$, we obtain $f(x_0) \leq f(z_x) + \| x_0 - z_x\| = r + \| x_0 - z_x\|$.
Thus, $f(x_0) = r + \| x_0 - z_x \|$.

\medskip

Combining $f(x_0) = r + \| x_0 - z_x \|$ and $f(z_x) = r$, the $1$--Lipschitz continuity of $f$ on $[x_0, z_x]$ readily yields
\[
    f \left( z_x+ t\frac{x_0-z_x}{\|x_0-z_x\|} \right)= f(z_x)+t  \quad \text{ for all } t\in[0,\|x_0-z_x\|].
\]
In particular
\[
    D_G f(z_x) \left( \dfrac{x_0 - z_x}{\| x_0 - z_x \|} \right) = 1. 
\]
Again applying Theorem~\ref{thm.straightline} at $z_x$, the uniqueness of $y_{z_x}$ implies that
\[
y_{z_x} = \dfrac{x_0-z_x}{\|x_0-z_x\|}
\]
and thus
\[
f\left(z_x + t y_{z_x}\right) = r+t \quad \text{for all $t>0$ such that $[z_x,z_x+t y_{z_x}]\subset \mathcal U$.}
\]
Hence, $f(x)= r+ \|x-z_x\| =r+\dist(x,[f=r])$.
The second part of the corollary follows directly from Theorem~\ref{thm.straightline} and the first part.
\hfill$\square$

\begin{remark}
    The distance representation in Corollary~\ref{corol.dist.2} generally fails when one of the two conditions in~\eqref{eq.cor.dist} fails.
    Indeed, consider $X = \R$ and $\mathcal U = (0, 1) \cup (1, 2)$.
Then, the function
\[
    f(x) = 
    \begin{dcases}
        x & x \in (0, 1) \\
        x + 100 & x \in (1, 2)
    \end{dcases}
\]
is differentiable in $\mathcal U$ and $|f'| \equiv 1$ on $\mathcal U$.
However, taking $r = 1/2$, we have $[f = r] = \{1/2 \}$.
At $\overline x = 3/2$, a direct computation gives
\[
    f(\overline x) = \dfrac{203}{2} >  \dfrac{1}{2} + 1= r + \mathrm{dist}(\overline x, [f = r]).
\] 
On the other hand, if we consider $X=(\R^2,\|\cdot\|_2)$, $\mathcal U:=\{x\in \R^2:~ x_1>1\}$ and $f: \mathcal U \to\R$ defined as the Euclidean norm $f(x)=\|x\|_2$. Then $f$ is differentiable on $\mathcal U$ and $\|Df\|\equiv 1$.
On the other hand, it can be easily checked that for $\overline{x}:=(2,6)$ and any $z\in [f=\sqrt{2}]$, $\mathrm{dist}(\overline{x},[f=\sqrt{2}])< \|\overline{x}-z\|_2$ and $f(\overline{x})= \sqrt{40}<\|\overline{x}-(1,1)\|_2 + \sqrt{2}=\sqrt{26}+\sqrt{2} = \mathrm{dist}(\bar x, [f = \sqrt{2}]) + \sqrt{2}$.
\end{remark}

\begin{comment}
\subsection{Proof of Corollary~\ref{corol.dist}}
Fix $x \in X$ and $r \in \R$.
Suppose first that $f(x) \geq r$.
Since $f$ is $1$--Lipschitz, it is not difficult to see that
\[
    f(x) - r \leq \mathrm{dist}(x, [f = r]).
\]
On the other hand, by Theorem~\ref{thm.straightline}, there exists a direction $y_x \in S_X$ such that
\[
    f(x + t y_x) = f(x) + t \quad \text{ for every } t \in \R.
\]
Therefore, setting $z_x := x - (f(x) - r)y_x$, we get $f(z_x) = r \in [f = r]$.
It follows that
\[
    \mathrm{dist}(x, [f = r]) \leq \| x - z_x \| = f(x) - r.
\]
Therefore, $f(x) - r = \mathrm{dist}(x, [f = r])$ for every $f(x) \geq r$. 
Analogously, we can prove that $f(x) - r = - \mathrm{dist}(x, [f = r])$ for every $f(x) \leq r$, which completes the proof.
\hfill$\square$
\end{comment}
\section{The finite--dimensional setting}\label{sec.4}
In finite dimensions, compactness of the unit sphere allows us to remove the uniform convexity assumption required in the infinite-dimensional setting.
In this section we prove Proposition~\ref{prop.fini-dim} and derive some consequences concerning gradient ranges and convex bodies.

\subsection{Proof of Proposition~\ref{prop.fini-dim}}
\emph{(i)}
The proof can be done as an adaptation of Corollary~\ref{corol.regu+linear}-\emph{(iii)} or see \cite[Theorem 5.3.7]{CS_2004}.
\medskip

\emph{(ii)}
Assume now that $\mathcal{U}=\R^d$. We show that $f$ is affine. 
The technique closely follows the one presented in \cite[Proposition 6.25]{TG_Thesis}. 
For the sake of completeness we provide a proof.

\medskip
  
Assume, without any loss of generality, that $f(0) = 0$ and let $r>0$. 
Set $x_r,y_r\in \overline{B}_r$ be such that
\[
x_r\in \arg\max\,\{f(x):~x\in \overline{B}_r\}\quad\text{and}\quad y_r\in \arg\min\,\{f(x):~x\in \overline{B}_r\}.
\]
As in Claim~\ref{claim.maxmin} of the proof of Theorem~\ref{thm.straightline}, we can ensure that $\|x_r\|=\|y_r\|=r$, $f(x_r)=r$ and $f(y_r)=-r$. 
    Let us define $\widehat{x}_r:=x_r/r$ and $\widehat{y}_r:=y_r/r$ and let $\widehat{x}_\infty, \widehat{y}_\infty\in \partial B_1$ as any accumulation point of the set $\{\widehat{x}_r\}_{r > 0}$ and $\{\widehat{y}_r\}_{r > 0}$, as $r$ tends to infinity, respectively.
    \begin{claim}\label{claim.Long_Ray}
        $f(t \widehat{x}_\infty)=t$ and $f(t\widehat{y}_\infty)=-t$ for all $t\geq 0$. 
    \end{claim}

    \textit{Proof of Claim~\ref{claim.Long_Ray}.}
   Since for all $r>0$ we have that $\|x_r\| = r$, $f(0)=0$, $f(x_r)=r$ and $f$ is $1$-Lipschitz, $f(t\widehat{x}_r)=t$ for all $t\in [0,r]$.
    The continuity of $f$ leads to the fact that  $f(t \widehat{x}_\infty)=t$ for all $t\geq 0$. 
    The case of $\widehat{y}_\infty$ is analogous. 
    The claim follows.
    \hfill$\Diamond$

    \medskip

    Due to the $1$-Lipschitz continuity of $f$, we also obtain that $\|\widehat{x}_\infty-\widehat{y}_\infty\|=2$. 
    Thus, there is a functional $u^*\in X^*$, such that $\|u^*\|=1$, $u^*(\widehat{x}_\infty)=1$ and $u^*(\widehat{y}_\infty)=-1$. 
    We finally claim that $f \equiv u^*$. 
    Indeed, let us start by fixing $z\in \ker(u^*)$. Through a Taylor expansion of the norm around $\widehat{x}_\infty$, we deduce that
    $\|r\widehat{x}_\infty-z\| = r\|\widehat{x}_\infty-z/r\|=r+ o(1)$. Since $\|r\widehat{x}_\infty-z\|\geq f(r\widehat{x}_\infty)-f(z)$, by sending $r$ to $+\infty$ we deduce that $f(z)\geq 0$. 
    A similar argument, but exchanging $r\widehat{x}_\infty$ by $r\widehat{y}_\infty$, leads to the fact that $f(z)\leq 0$. Thus, $f(z)=0$. So, $f(\ker(u^*))=\{0\}$.
    Now, to finish the proof it is enough to show that $f(z+t\widehat{x}_\infty)=t$ and $f(z+t\widehat{y}_\infty)=-t$ for all $z\in \ker (u^*)$ and $t\geq 0$. 
    Fix $z\in \ker (u^*)$ and $t>0$. Since $f$ is $1$-Lipschitz, $f(z+t\widehat{x}_\infty)\leq t$
    . As in the previous part of the proof, we have that     \[
    \| r\widehat{x}_\infty-(z+t\widehat{x}_\infty)\|= r-t + o(1) \quad \text{ as } r \longrightarrow + \infty,
    \]
    which allows us to deduce that $f(z+t\widehat{x}_\infty)\geq t$. Thus, $f(z+t\widehat{x}_\infty)=t$.
    The arguments for the case $z+t\widehat{y}_\infty$ are analogous.
    Since any $x \in \R^d$ can be written as $x = z + t \widehat{x}_\infty$ if $u^\ast(x) \geq 0$ or as, $x = z + t \widehat{y}_\infty$ if $u^\ast(x) < 0$ for some $z \in \mathrm{ker}(u^\ast)$ and some $t \geq 0$, we conclude that $f(x) = u^\ast(x)$ for every $x \in \R^d$.
    This completes the proof.
\hfill$\square$

\subsection{Proof of Corollary~\ref{cor.convexaffine}}

Several authors have studied the range of the differential for differentiable maps, with particular emphasis on bump functions defined on arbitrary Banach spaces, for instance, see \cite{AD_2001, AJD_2003, BFL_2002}.
In our case, we will focus on the case when the range of the differential, $D f(\R^d)$, is somehow small.

\medskip

Considering the canonical Euclidean norm of $\R^d$, Proposition~\ref{prop.fini-dim} readily yields the following corollary (this remark was already noticed in \cite{CC_2010}).
\begin{corollary}\label{cor.first example}
    Let $f:\R^d\to\R$ be a differentiable map such that $D f(\R^d)\subset S^{d-1}$. Then $f$ is affine.
\end{corollary}

In the sequel, and until the end of this section, $\R^d$ will be always equiped with its canonical Euclidean norm.
Before continuing, let us recall that a convex body $C\subset \R^d$ is a nonempty closed convex set such that $C=\overline{\mathrm{int}\, C}$. 
Also, if $0\in \mathrm{int}\, C$, the Minkowski gauge associated with $C$, $\rho_C:\R^d\to\R$, is the (convex) positively homogeneous map defined by
\[\rho_C(x):=\inf\{\lambda>0:~ x\in \lambda C\},\quad\text{for all }x\in \R^d.\]
Below we use some standard results in convex analysis. We refer the reader to \cite{Rock_70}

\begin{proposition}\label{prop.asym}
    Let $K\subset \R^d$ be a compact convex body such that $0\in \mathrm{int} K$ and $\partial K$ contains no nontrivial segment. Then, if $f:\R^d\to\R$ is differentiable and $D f(\R^d)\subset\partial K$, then $f$ is affine.
\end{proposition}

% \begin{remark}
% The above proposition also yields a local r egularity statement, whose proof follows as in Proposition~\ref{prop.fini-dim}--(i), replacing the norm by the Minkowski gauge.
%More precisely,  let $K\subset \R^d$ be a compact convex body such that $0\in \mathrm{int} K$ and $\partial K$ contains no nontrivial segment.
%Let $\mathcal U \subset \R^d$ be an open subset and $f : \mathcal U \to R$ be a differentiabl function with $D f(\mathcal U) \subset \partial K$.
%Then, $f \in C^{1, \omega}_{\mathrm{loc}}(\mathcal U)$, where $\omega$ is the concave modulus of continuity of $D \rho_{K^\circ}$ on $\partial K^\circ$ and $K^\circ$ is the polar set of $K$.
% \end{remark}

\begin{proof}
The proof of Proposition~\ref{prop.asym} follows the lines of the proofs of Theorem~\ref{thm.straightline} and Proposition~\ref{prop.fini-dim}. Therefore, below we only highlight the main differences that need to be taken into account to prove Proposition~\ref{prop.asym}.

\medskip

Let $Q\subset \R^d$ be the compact convex body defined by \[Q:=\{x\in \R^d:~ \langle y,x\rangle\leq 1\quad\text{for all }y\in K\}.\]
That is, $Q:=K^\circ$, the polar set of $K$.
Since $K$ is a compact convex body with $0\in\mathrm{int}\, K$ and $\partial K$ contains no nontrivial segment, $Q$ is also a compact convex body with $0\in \mathrm{int}\, Q$ and $\partial Q$ is smooth.
Moreover, the Minkowski gauge $\rho_K$ can be regarded as the dual (asymmetric) norm associated with $\rho_Q$ (and vice versa) in the following sense
\[\rho_K(y)=\sup_{q\in Q}\langle q,y\rangle=\max_{q\in \partial Q}\langle q,y\rangle,\quad\text{for all }y\in \R^d.\]

Let $f:\R^d\to\R$ be a differentiable function such that $D f(\R^d)\subset\partial K$. So, 
\begin{equation}\label{eiko_K}
\rho_K( D f(x))=1,\quad\text{for all }x\in \R^d.
\end{equation}

\medskip
\begin{claim}\label{claim.a}
    For every $r>0$ and $\overline{x}\in \R^d$, $\max_{x\in rQ} f(\overline{x}+x)=f(\overline{x})+r$.
\end{claim}
%\emph{For every $r>0$ and $\overline{x}\in \R^d$, $\max_{x\in rQ} f(\overline{x}+x)=f(\overline{x})+r$.} 
\textit{Proof of Claim~\ref{claim.a}.}
Fix $r>0$ and $\varepsilon>0$.
Without loss of generality, assume that $\overline{x}=0$ and $f(0)=0$. 
Due to~\eqref{eiko_K}, for every $x \in \R^d$, there are $q_x\in \partial Q$ and $t_x>0$ such that
\[f(x+t q_x)-f(x)> (1-\varepsilon)t=(1-\varepsilon)\rho_Q((x+tq_x)-x),\quad\text{for all }t\in(0,t_x).\]
Therefore, for any $x\in \mathrm{int}(rQ)$, there is $x'\in rQ$ such that 
\[f(x')-f(x)>(1-\varepsilon)\rho_Q(x'-x).\]
Hence, using transfinite induction as in the proof of Theorem~\ref{thm.straightline}, we can show that
\[
\sup_{x\in rQ} f(x)\geq (1-\varepsilon)r.
\]
After sending $\varepsilon$ to $ 0$, we get $\sup_{x\in rQ} f(x)\geq r$.
Also, by continuity of $f$, the supremum is attained at some $y_r\in r\partial Q$. Set $q_r:=y_r/r$.
Now, since 
\[\langle D f(tq_r),q_r\rangle \leq 1,\quad\text{for all }t\in [0,r],\]
we deduce that $f(y_r)=f(rq_r)\leq r$.
Thus, $f(y_r)=r$.
    \hfill$\Diamond$
\medskip

\begin{claim}
    \label{claim.b}
    For every $\overline{x}\in \R^d$, there is $y_{\overline{x}}\in \partial Q$ such that $f(\overline{x}+ty_{\overline{x}})= f(\overline{x}) + t$ for all $t\geq 0$.
\end{claim}
%\textbf{Claim:} 
%\emph{For every $\overline{x}\in \R^d$, there is $y_{\overline{x}}\in \partial Q$ such that $f(\overline{x}+ty_{\overline{x}})= f(\overline{x}) + t$ for all $t\geq 0$.}
\textit{Proof of Claim~\ref{claim.b}.} It follows exactly as in the proof of Proposition~\ref{prop.fini-dim} by using the compactness of $\partial Q$ and the continuity of $f$. Indeed, for $\overline{x}:=0$ and using the notation of the previous claim, we deduce that $f(tq_r)=t$ for all $t\in [0,r]$. Thus, for any $y_0\in \partial Q$, accumulation point of the set $\{q_r\}_{r>0}$ as $r$ goes to infinity, we have that $f(ty_0)=t$ for all $t>0$.
    \hfill$\Diamond$
\medskip

\begin{claim}\label{claim.c}
    For every $\overline{x}\in \R^d$, there is $z_{\overline{x}}\in \partial Q$ such that $f(\overline{x}-tz_{\overline{x}})=f(\overline{x})-t$ for all $t\geq 0$.
\end{claim}
%\textbf{Claim:}
%\emph{For every $x\in \R^d$, there is $z_x\in \partial Q$ such that $f(x-tz_x)=f(x)-t$ for all $t\geq 0$.}
\textit{Proof of Claim~\ref{claim.c}.}
Indeed, setting $\widehat{f}(x):=-f(-x)$ we deduce that $\rho_K(D \widehat{f}(x))=\rho_K(D f(-x))=1$, for every $x\in \R^d$. 
Therefore, for all $\overline{x}\in \R^d$ there is $w_{\overline{x}}\in \partial Q$ such that $\widehat{f}(\overline{x}+tw_{\overline{x}})=\widehat{f}(\overline{x})+t$, for all $t\geq 0$. 
Thus, setting $z_{\overline{x}}:=w_{-\overline{x}}$ we have
\[f(\overline{x}-tz_{\overline{x}})=f(\overline{x})-t,\quad\text{for every }t\geq 0.\]
    \hfill$\Diamond$
    
\begin{claim}
    \label{claim.d}
    For every $\overline{x}\in\R^d$, $\langle D f(\overline{x}),y_{\overline{x}}\rangle = \langle D f(\overline{x}),z_{\overline{x}}\rangle =1$ and $D \rho_Q(y_{\overline{x}})= D \rho_Q(z_{\overline{x}})=D f(\overline{x})$.
\end{claim}
%\textbf{Claim:} 
%\emph{It holds $\langle D f(x),y_x\rangle = \langle D f(x),z_x\rangle =1$ and $D \rho_Q(y_x)= D \rho_Q(z_x)=D f(x)$.}
\textit{Proof of Claim~\ref{claim.d}.}
Indeed, because
\[\langle D f(\overline{x}),y_{\overline{x}}\rangle =\lim_{t\to0^+} \dfrac{f(\overline{x}+ty_{\overline{x}})-f(\overline{x})}{t}=1=\lim_{t\to 0^+}\dfrac{f(\overline{x}-tz_{\overline{x}})-f(\overline{x})}{-t} = \langle D f(\overline{x}),z_{\overline{x}}\rangle.\]
Finally, the smoothness of $\partial Q$ implies that $D f(\overline{x})$ is the only vector $w\in K$ such that $\langle w,y_{\overline{x}}\rangle=1$ (and $\langle w,z_{\overline{x}}\rangle =1$). Thus, $D \rho_Q(y_{\overline{x}})=D \rho_Q(z_{\overline{x}})=D f(\overline{x})$.
 \hfill$\Diamond$
\medskip

\textbf{Conclusion:} \emph{$f$ is affine.} 
Indeed, having both directions $\{y_0,z_0\}\subset \partial Q$ such that ${f(ty_0)=t}$ and ${f(-tz_0)=-t}$ for all $t\geq 0$, that $D\rho_Q(y_0)=D\rho_Q(z_0)=Df(0)$ and the fact that
\[
f(v)-f(w)\leq \rho_Q(v-w),\quad\text{for all }v,w\in \R^d,
\]
we can mimic the proof of Corollary~\ref{corol.regu+linear} $(i)$ to deduce that $f$ must be affine.
\end{proof}

Now we are ready to prove Corollary~\ref{cor.convexaffine}.

\begin{proof}[Proof of Corollary~\ref{cor.convexaffine}]\phantom{tri}\\
$(1)\Longrightarrow(2)$. The proof follows exactly the same idea as in Example~\ref{ex: nonlinear solution}. 
Let $x,y\in \partial C$ be such that the segment $[x,y]\subset \partial C$. Let $\phi:\R\to \R$ be any non--affine differentiable map such that $\phi'(t)\in [-1/2, 1/2]$ and consider $f:\R^d\to\R$ be defined by
\[f(z):=\left\langle\dfrac{x+y}{2},z \right\rangle + \phi(\langle x-y,z\rangle),\quad\text{for all }z\in \R^d\]
It follows that $f$ is differentiable and $D f(\R^d)\subset [x,y]\subset \partial C.$
\medskip \newline 
$ (2)\Longrightarrow (1)$.
Reasoning by  a contrapositive argument, we prove that $\neg (1)\implies \neg (2)$.
    Let $\overline{x}\in \mathrm{int}\,C$ and set $K:=C - \bar x$. Then, $K$ is a compact convex set containing $0$ in its interior. 
    %Let us now consider $Q := K^\circ$, the polar set of $K$.
    % , that is, $Q:=\{x\in \R^d:~ \langle x,y \rangle\leq 1,~\text{for all }y\in K\}.$
   % It follows that $Q$ is a convex set. Moreover, since $\partial K$ does not contain nontrivial segments, we have that $\partial Q$ is smooth. 
    %Notice that $Q$ is a convex set and $\partial Q$ is smooth.
    %Therefore, the (asymmetric) norm on $\R^d$ defined as the Minkowski gauge $\rho_Q$ associated with $Q$ is $C^1$-smooth on $\R^d\setminus\{0\}$.
    %We also have the following duality relation between the Minkowski gauge associated with $K$ and $Q$:
    %\[
    %\rho_{Q, \ast}(z):=\sup\{\langle x,z\rangle:~\rho_Q(x)\leq 1\}=\inf\{\lambda>0:~ z\in \lambda K\} =:\rho_K(z).\]
    %Consider now a differentiable function $f:\R^d\to\R$ such that $D f(\R^d)\subset \partial C$. Then, the function $g:=f-\langle \overline{x},\cdot\rangle$ satisfies that $D g(\R^d)\subset \partial K$.
    Let $f:\R^d\to\R$ be a differentiable function such that $Df(\R^d)\subset \partial C$. Then, the function $g:=f-\langle \overline{x},\cdot\rangle$ is differentiable function that satisfies $Dg(\R^d)\subset \partial K$.
    Thus, Proposition~\ref{prop.asym} implies that $g$ is affine. 
    Hence, $f$ is affine as well.
\end{proof}
\appendix

\section{Appendix}
For the sake of completeness, let us record a standard criterion for a $C^{1, \omega}$--estimate, see also \cite[Proposition 2.6]{AM_2021} and \cite[Theorem 3.3.7]{CS_2004}.
Observe that the smoothness of the norm is not needed in the following result.

\begin{proposition}\label{prop.C1ome}
Let $(X,\|\cdot\|)$ be a Banach space and let $\mathcal U \subset X$ be a nonempty open set. 
Let $\omega:[0,+\infty)\to[0,+\infty)$ be a concave modulus of continuity. Let $f : \mathcal{U} \to \mathbb{R}$ be a G\^ateaux-differentiable function.
Assume that, for every $\bar{x} \in \mathcal{U}$, there are $r = r_{\bar{x}} > 0$ and $C>0$ such that $B(\bar{x},4r) \subset \mathcal{U}$ and
\begin{equation}\label{eq:uni-re}
\left| f(x+h) - f(x) - \langle D_Gf(x),h \rangle \right| \leq C \|h\|\omega(\|h\|)
\end{equation}
for every $x\in B(\bar{x},2r)$ and every $h \in X$ with $\|h\| \leq 2r$.
Then, $f$ is Fr\'echet differentiable in $\mathcal U$ with $Df = D_Gf$ and for every $\bar x \in \mathcal U$, one has
\begin{equation}\label{eq:deri-mol}
\| Df(y) - Df(z) \|_\ast \leq 6C \omega(\|y-z\|)
\end{equation}
for every $y,z \in B(\bar{x},r_{\bar x} /2)$. 
In particular, $f \in C_{\mathrm{loc}}^{1,\omega}(\mathcal U)$.
\end{proposition}

\begin{proof}
Let $\bar x \in \mathcal{U}$ and let $r > 0$ such that $B(\bar x, 4r) \subset \mathcal{U}$ and ~\eqref{eq:uni-re} holds.
Fix $x\in B(\bar{x},2r)$. 
It follows from~\eqref{eq:uni-re} that
\begin{equation*}
\frac{ \left| f(x+h) - f(x) - \langle D_G f(x), h \rangle \right| }{ \|h\| } \leq C\omega(\|h\|).
\end{equation*}
Since $\omega(t)\to 0$ as $t\to 0^+$, we deduce that $f$ is Fr\'echet differentiable at $x$ and $Df(x) = D_Gf(x)$.

\medskip

It remains to check~\eqref{eq:deri-mol}. 
For $x\in B(\bar{x},2r)$ and $\| h \| \leq 2r$, set $R_x(h) := f(x + h) - f(x) - \langle Df(x), h\rangle$. 
Fix $y,z \in B(\bar{x},r/2)$  with $y \neq z$ and $v \in S_X$. 
Set $s := \| y - z \| > 0$.
A direct computation  yields
\begin{equation}\label{eq:re-id}
s \langle Df(z)-Df(y), v \rangle = R_y(z - y + sv) - R_y(z - y) - R_z(sv).
\end{equation}
Observe that $s \leq r$ and so $\| z - y + sv \| \leq 2s \leq 2r$. 
It follows from~\eqref{eq:uni-re} that $|R_x(h)|\leq C\|h\|\omega(\|h\|)$ for every $x \in B(\bar x, 2r)$ and $\| h \| \leq 2r$.
Therefore, 
\begin{align*}
s\left|\langle Df(z) - Df(y), v\rangle\right| &\leq C\|z - y + sv\|\omega(\|z - y + sv\|) + C\|z - y\|\omega(\|z - y\|) + Cs\omega(s) \\
&\leq 2Cs\omega(2s) + 2Cs\omega(s).
\end{align*}
Since $\omega$ is concave and $\omega(0) = 0$, we have $\omega(2s)\leq 2\omega(s)$. 
Consequently,
\begin{equation*}
\left|\langle Df(z)-Df(y),v\rangle\right| \leq 6C\omega(s).
\end{equation*}
Taking the supremum over $v\in S_X$, we conclude that
\[
    \| Df(y) - Df(z) \|_\ast \leq 6C \omega(\| y - z \|) \quad \text{ for every } y, z \in B(\bar x, r/2).
\]
Therefore, $f\in C_{\mathrm{loc}}^{1,\omega}(\mathcal U)$, which completes the proof.

\end{proof}
\textbf{Disclosure statement:} The authors report there are no competing interests to declare.\\
\textbf{Declaration of generative AI use:} The authors report that they used generative AI (ChatGPT) only for English-language polishing.

\vspace{0.5cm}

\noindent Tr\'i Minh L\^E
\medskip 

\noindent  Faculty of Mathematics,
University of Vienna
\newline Oskar-Morgenstern-Platz 1, 1090 Wien
\medskip
\newline\noindent E-mail: \texttt{tri.minh.le@univie.ac.at}
\newline\noindent\texttt{https://sites.google.com/view/tri-minh-le}
\newline\noindent\texttt{https://orcid.org/0009-0007-7396-1987}
\smallskip\newline\noindent
\noindent Research supported by the Austrian \texttt{FWF} grant \texttt{DOI 10.55776/STA223}.

\vspace{0.5cm}

\noindent Sebasti\'an TAPIA-GARC\'IA

\medskip

\noindent Institut f\"{u}r Stochastik und Wirtschaftsmathematik, VADOR E105-04
\newline TU Wien, Wiedner Hauptstra{\ss }e 8, A-1040 Wien\medskip
\newline\noindent E-mail: \texttt{sebastian.tapia.garcia@tuwien.ac.at}
\newline\noindent\texttt{https://sites.google.com/view/sebastian-tapia-garcia/}
\newline\noindent\texttt{https://orcid.org/0000-0002-2667-4279}

\medskip

\noindent Research partially supported by the grants: \smallskip\newline 
Austrian Science Fund (FWF P-36344N) (Austria)\\
SABOCPR project ANR-25-CE40-3469-01 / FWF 4368225 (France-Austria)
\end{document}